\documentclass[12pt, letterpaper]{amsart}

\usepackage{graphicx}
\usepackage{lipsum}

\calclayout

\usepackage{latexsym,amssymb}
\usepackage{amsmath}
\usepackage{amsthm} 
\usepackage[mathscr]{eucal}
\usepackage{color}
\usepackage[abbrev, nobysame]{amsrefs}
\usepackage[T1]{fontenc}

\usepackage{latexsym}
\usepackage{amssymb}
\usepackage{amsfonts}
\usepackage{enumerate}
\usepackage{mathrsfs}
\usepackage{color}
\usepackage{comment}
\usepackage{mathtools}
\usepackage{slashbox}
\usepackage{here}

\mathtoolsset{showonlyrefs=true}

\usepackage{amsaddr}

\newtheorem{thm}{Theorem}
\newtheorem*{thm*}{Theorem}

\theoremstyle{plain}
\newtheorem{Theorem}{\bf Theorem}[section]
\newtheorem{Lemma}{\bf Lemma}[section]
\newtheorem{Proposition}{\bf Proposition}[section]
\newtheorem{Corollary}{\bf Corollary}[section]
\newtheorem{Remark}{\bf Remark}[section]
\newtheorem{Example}{\bf Example}[section]
\newtheorem{Definition}{\bf Definition}[section]

\newenvironment{lemma}{\begin{Lemma}$\!\!\!$}{\end{Lemma}}
\newenvironment{proposition}{\begin{Proposition}$\!\!\!$}{\end{Proposition}}

\newenvironment{definition}{\begin{Definition}$\!\!\!$}{\end{Definition}}

\numberwithin{equation}{section}

\newcommand{\di}{\mathop{}\!\mathrm{d}}

\DeclareMathOperator{\supp}{\operatorname{supp}}

\newcommand{\dist}{\mathsf{d}}
\DeclareMathOperator{\Dist}{dist}

\def\XXint#1#2#3{{\setbox0=\hbox{$#1{#2#3}{\int}$}
\vcenter{\hbox{$#2#3$}}\kern-.5\wd0}}

\begin{document}
	
	\title[semilinear heat equations in the very weak framework]
	{Necessary conditions for the solvability of\\ fractional semilinear heat equations\\ in the very weak framework}
	
	
	\author{Kotaro Hisa}
	\address[Kotaro Hisa]{Graduate School of Mathematical Sciences, The University of Tokyo, 3-8-1 Komaba,
Meguro-ku, Tokyo 153-8914, Japan.}
	\email{hisak@ms.u-tokyo.ac.jp}
	

	
	
	
	
	
	\begin{abstract}
In this paper we obtain necessary conditions  on the initial
value for the solvability of the Cauchy problem for semilinear heat equations.
These necessary conditions were already obtained in the  framework of integral solutions, but not in that of very weak ones.
We  establish a new proof method, which can derive the desired conditions in the framework of very weak solutions.
In particular, since any integral solution is a very weak solution, our conditions are more general.
	\end{abstract}
	\maketitle
	
\section{Introduction}
Consider the Cauchy problem for the fractional semilinear heat equation
\begin{equation*}
	{\rm (SHE)}\qquad
	\left\{
	\begin{aligned}
	&\partial_t u + (-\Delta)^\frac{\theta}{2} u = u^p,\quad && x\in\mathbb{R}^N\,\,\, t\in(0,T), \\
	&u\ge 0\quad && \mbox{a.e. in} \quad  \mathbb{R}^N\times (0,T), \\
	&u(0) = \mu && \mbox{in}\quad \mathbb{R}^N,
	\end{aligned}
	\right.
\end{equation*}
where  $N\geq1$,  $\theta\in(0,2]$, $p>1$, and
$\mu$ is a nonnegative Radon measure on $\mathbb{R}^N$.
Here, $(-\Delta)^{\theta/2}$ denotes the fractional power of the Laplacian $-\Delta$ in $\mathbb{R}^N$.
Let us recall the definition of the fractional Laplacian $(-\Delta)^{\theta/2}$.
Let  $\mathcal{S}(\mathbb{R}^N)$ be  the  Schwartz space in $\mathbb{R}^N$.
$\mathcal{F}$ denotes the Fourier transform on $\mathbb{R}^N$.
For $\theta\in(0,2)$, we define the  fractional Laplacian $(-\Delta)^{\theta/2}$ as the operator given by the Fourier multiplier $|\xi|^\theta$, that is, for $u\in \mathcal{S}(\mathbb{R}^N)$
\[
\mathcal{F}[(-\Delta)^\frac{\theta}{2}u](\xi) = |\xi|^\theta \mathcal{F}[u](\xi).
\]
By  simple calculations, we give another expression as an
integral operator, that is,
\[
(-\Delta)^\frac{\theta}{2} u(x) = A(N,\theta) P.V. \int_{\mathbb{R}^N}\frac{u(x)-u(y)}{|x-y|^{N+\theta}} \, \di y,
\]
where $A(N,\theta)>0$ is a normalizing constant
\[
A(N,\theta) = \frac{\Gamma\left({(N+\theta)/2}\right)}{\pi^{\frac{N}{2}} \Gamma\left(2-{\theta/2}\right)} \cdot\frac{\theta}{2} \left(1-\frac{\theta}{2}\right)
\]
and $\Gamma$ is the Gamma function (see {\it e.g.} Di~Nezza--Palatucci--Valdinoci \cite{DPV12}).
In this paper we show necessary conditions on $\mu$ for the existence of nonnegative solutions of problem 
{\rm (SHE)} in the very weak framework.

For $x\in \mathbb{R}^N$ and $r>0$, set $B(x,r) :=\{y\in\mathbb{R}^N; |x-y|<r\}$
and
define $|B(x,r)|$ as the volume of $B(x,r)$.
Let us recall related studies.
This study began with Baras--Pierre \cite{BP85} in 1985 and has been conducted in \cites{FHIL23, FHIL24,HI18,IS19,IKO20, TY23}.
These were conducted in either of the following frameworks:
\begin{definition}
\label{Definition:WS}
Let $N\ge 1$, $\theta\in(0,2]$, and $p>1$.
Let $u$ be a nonnegative function in  $\mathbb{R}^N\times(0,T)$, where $T\in(0,\infty]$.
We say that $u$ is an integral solution to problem {\rm (SHE)} in $\mathbb{R}^N\times[0,T)$ if $u$ satisfies
\begin{equation}
\label{eq:defisol}
\begin{split}
\infty>u(x,t)
& = \int_{\mathbb{R}^N}\Gamma_\theta(x-y,t)\, \di \mu(y)\\
&+ \int_0^t\int_{\mathbb{R}^N} \Gamma_\theta(x-y,t-s) u(y,s)^p\,\di y\di s
\end{split}
\end{equation}
for almost all $x\in\mathbb{R}^N$ and $t\in(0,T)$.
If $u$ satisfies \eqref{eq:defisol} with $=$ replaced by $\ge$, then we say that $u$ is an integral supersolution in 
$\mathbb{R}^N\times[0,T)$.
Here, $\Gamma_\theta$ is the fundamental solution to
\[
\partial_t v + (-\Delta)^\frac{\theta}{2} v = 0 \quad \mbox{in} \quad \mathbb{R}^N\times(0,\infty).
\]
\end{definition}

\begin{definition}
\label{Definition:IS}
Let $N\ge 1$, $\theta\in(0,2]$, and $p>1$.
Let $u$ be a nonnegative function in  $\mathbb{R}^N\times(0,T)$, where $T\in(0,\infty]$.
We say that $u$ is a very weak solution to problem {\rm (SHE)} in $\mathbb{R}^N\times[0,T)$ if   $u$ satisfies
$u^p\in L^1_{\rm loc} (\mathbb{R}^N\times[0,T))$ and
\begin{equation}
\label{eq:defwsol}
\begin{split}
&\int_{0}^\tau \int_{\mathbb{R}^N} u(-\partial_t + (-\Delta)^\frac{\theta}{2})\phi\,\di x\di t= \int_0^\tau \int_{\mathbb{R}^N} u^p\phi \,\di x\di t + \int_{\mathbb{R}^N} \phi(0)\, \di\mu 
\end{split}
\end{equation}
for any $\tau\in (0,T)$ and  any compactly supported $\phi\in C^{\infty}(\mathbb{R}^N\times[0,\tau])$ with $\phi(\tau)=0$.
If $u$ satisfies \eqref{eq:defwsol} with $=$ replaced by $\ge$, then we say that $u$ is a very weak supersolution in 
$\mathbb{R}^N\times[0,T)$.
\end{definition}
The relationship between the two definitions above is as follows:
\begin{proposition}
\label{Proposition:IisQ}
If $u$ is an integral solution to {\rm (SHE)} in $\mathbb{R}^N\times[0,T)$, where $T>0$,
then $u$ is also a very weak solution to {\rm (SHE)} in $\mathbb{R}^N\times[0,T)$.
\end{proposition}
By this proposition we see that the framework of very weak solutions is a more general concept than that of integral ones.
To make this paper self contained, 
the proof of this proposition is given in Appendix below.
Summarizing previous studies, the following facts have already been obtained.
For $N\ge1$ and $\theta\in (0,2]$,  denote
\[
p_{\theta, N} := 1 + \frac{\theta}{N}.
\]

\begin{thm}
\label{Theorem:NS}
Let $N\ge1$, $\theta\in(0,2]$, and $p>1$.
Assume that problem {\rm (SHE)} possesses a supersolution in $\mathbb{R}^N\times[0,T)$ in the sense of Definitions {\rm \ref{Definition:WS}} or {\rm \ref{Definition:IS}}, where $T\in(0,\infty)$.
Then there exists $\gamma>0$ depending only on $N$, $\theta$, and $p$ such that
\begin{itemize}
\item[(i)] If $p\neq p_{\theta,N}$, then $\displaystyle{\sup_{z\in\mathbb{R}^N} \mu(B(z,\sigma))\le\gamma\sigma^{N-\frac{\theta}{p-1}}}$ for all $0<\sigma<T^\frac{1}{\theta}$;
\item[(ii)] If $p=p_{\theta,N}$, then $\displaystyle{\sup_{z\in\mathbb{R}^N} \mu(B(z,\sigma))\le\gamma \left[\log\left(e+\frac{T^\frac{1}{\theta}}{\sigma}\right)\right]^{-{\frac{N}{\theta}}}}$ for all $0<\sigma<T^\frac{1}{\theta}$.
\end{itemize}
\end{thm}
We already know from \cites{FHIL23, HI18, IKO20, TY23} that these conditions are optimal.
However, in the case of $\theta\in(0,2)$, Theorem \ref{Theorem:NS} was obtained only in the framework of integral solutions and not in that of  very weak ones (see the following table).

\begin{table}[H]
 \begin{center}
   \caption{Previous studies on necessary conditions for the solvability.}
  \begin{tabular}{|l||r|r|r|}
  \hline
   	&Integral solution  		& Very weak solution							 	\\\hline\hline
    $\theta =2$ 			& \cites{FHIL23, HI18}			&\cites{BP85, IS19, IKO20, TY23}								\\\hline
    $\theta\in(0,2)$ 			&\cites{FHIL23, HI18}			&$\times$								\\\hline
  \end{tabular}
 \end{center}
\end{table}

Since  necessary conditions for the solvability are  conditions for initial value $\mu$ when assuming the existence of  solutions, it can be argued that the weaker the definition of solutions, the more general  theory.
Therefore, the goal of this paper is as follows:

\begin{itemize}
\item {\it We establish a new proof method which can deal with the case of  $\theta\in(0,2)$ to  prove Theorem~{\rm \ref{Theorem:NS}} in the framework of very weak solutions.}
\end{itemize}
Our proof in the case of $p\neq p_{\theta,N}$ is a simple extension of the arguments in \cites{FK12, GK01,  KQ02}, which deals with the nonexistence of global-in-time solutions, whereas the case of $p=p_{\theta, N}$ is completely different from previous studies \cites{BP85,FK12, GK01, IS19, IKO20,  KQ02, TY23}.
We provide a proof by employing a solution to a certain  adjoint heat equation 
$(-\partial_t+(-\Delta)^{\theta/2}) \phi =f$
with a suitable external force term $f$, as a test function.
This method is often used to prove the comparison principle of linear (also semilinear) heat equations 
and the uniqueness of solutions, 
and we can find, for {\it e.g.}, in \cite[Proposition~52.13]{QS19} (in the case of $\theta=2$),
\cite[Theorem~5.1]{BSV17} (in the case of $\theta\in (0,2)$), and 
also in \cites{AMY22, BPSV14, FK13, HM24, LPPS15}.
One of the novelties of this paper lies in construction of an external force term $f$ of an adjoint heat equation such that the estimate in Theorem~\ref{Theorem:NS} (ii) can be derived.
This method of construction is inspired by \cite{HIT20}, which  dealt with the fractional semilinear heat equation with an inhomogeneous term,
\[
\partial_t v + (-\Delta)^\frac{\theta}{2} v = v^q +\nu, \quad x\in\mathbb{R}^N,\,\, t>0, \qquad v(0)=0 \quad\mbox{in} \quad\mathbb{R}^N,
\]
where $N\ge1$, $\theta\in(0,2]$, $q>1$, and $\nu$ is a nonnegative Radon measure on $\mathbb{R}^N$.
The paper \cite{HIT20} employed a cut-off function of $t$ and a solution to a certain fractional Poisson equation
$(-\Delta)^{\theta/2} \phi =f$
with a  suitable external force term $f$,
 as a test function, which was used to obtain an optimal estimate of $\nu$, such that it corresponds to Theorem~\ref{Theorem:NS} (ii).
In the present paper, we extend this method to an adjoint heat equation from the Poisson equation.

The rest of this paper is organized as follows. 
In Section~2 we prepare a basic inequality and construct a test function.
In Section~3 we prove  Theorem~\ref{Theorem:NS}.
In Section~4 we prove  Proposition~\ref{Proposition:IisQ}.
\section{Construction of a test function.}
In what follows, the letter $C$  denotes a generic positive constant depending only on $N$, $\theta$, and $p$.
In this section we prepare a basic inequality and construct a test function.
Hereafter, unless otherwise mentioned, the case of $\theta\in(0,2)$ shall be considered.
The case of $\theta=2$ is simpler to prove and almost identical to the  case of $\theta\in(0,2)$.
The following lemma is the key to the proof of our main theorem.
\begin{lemma}
\label{Lemma:kihon}
Let $N\ge1$ and $\theta\in(0,2]$.
Assume that problem {\rm (SHE)} possesses a very weak  supersolution in $\mathbb{R}^N\times[0,T)$, where $T\in(0,\infty)$.
Then there exists a constant $C>0$ depending only on $N$, $\theta$,  and $p$ such that
\begin{equation}
\label{eq:kihon}
\int_{\mathbb{R}^N} \varphi(0)^\frac{p}{p-1}\, \di\mu \le C \int\int_{\supp\,\varphi} \left|(-\partial_t + (-\Delta)^\frac{\theta}{2}) \varphi\right|^\frac{p}{p-1} \, \di x\di t
\end{equation}
for $0\le \varphi\in C^{\infty}_0(\mathbb{R}^N\times[0,\tau])$ with $\varphi(\tau)=0$, where  $\tau\in(0,T)$.
\end{lemma}
\begin{proof}
Let $\varphi$ be such that $0\le \varphi\in C^{\infty}_0(\mathbb{R}^N\times[0,\tau])$ with $\varphi(\tau)=0$.
Then \eqref{eq:defwsol} with $\phi = \varphi^{p/(p-1)}$ yields
\begin{equation}
\label{eq:2.1}
\begin{split}
0&\le \int_0^T \int_{\mathbb{R}^N} u^p\varphi^\frac{p}{p-1} \,\di x\di t + \int_{\mathbb{R}^N} \varphi(0)^\frac{p}{p-1}\, \di\mu \\
&\le\int_{0}^T \int_{\mathbb{R}^N} u(-\partial_t + (-\Delta)^\frac{\theta}{2})\varphi^\frac{p}{p-1}\,\di x\di t\\
&\le \frac{p}{p-1}\int\int_{\supp\,\varphi} u \varphi^\frac{1}{p-1} (-\partial_t + (-\Delta)^\frac{\theta}{2})\varphi\,\di x\di t\\
&\le C\int\int_{\supp\,\varphi} u \varphi^\frac{1}{p-1} \left|(-\partial_t + (-\Delta)^\frac{\theta}{2})\varphi\right|\,\di x\di t.\\
\end{split}
\end{equation}
Here, we used 
\[
(-\Delta)^\frac{\theta}{2} \varphi^\frac{p}{p-1} \le \frac{p}{p-1} \varphi^\frac{1}{p-1} (-\Delta)^\frac{\theta}{2} \varphi.
\]
See \cite[Appendix]{FK12} and \cite[Proposition~3.3]{Ju05}.
It follows from \eqref{eq:2.1} and Young's inequality that
\begin{equation*}
\begin{split}
&\int_0^T \int_{\mathbb{R}^N} u^p\varphi^\frac{p}{p-1} \,\di x\di t + \int_{\mathbb{R}^N} \varphi(0)^\frac{p}{p-1}\, \di\mu \\
&\le C\int\int_{\supp\,\varphi} u \varphi^\frac{1}{p-1} \left|(-\partial_t + (-\Delta)^\frac{\theta}{2})\varphi\right|\,\di x\di t.\\
&\le \int_0^T \int_{\mathbb{R}^N} u^p\varphi^\frac{p}{p-1} \,\di x\di t + 
C\int\int_{\supp\,\varphi}   \left|(-\partial_t + (-\Delta)^\frac{\theta}{2})\varphi\right|^\frac{p}{p-1} \,\di x\di t.
\end{split}
\end{equation*}
Then we obtain the desired inequality and the proof is complete.
\end{proof}

For $\theta \in (0,2)$, define
\[
H^\frac{\theta}{2}(\mathbb{R}^N) := \{u\in L^2(\mathbb{R}^N); |\xi|^\frac{\theta}{2}\mathcal{F}[u](\xi)\in L^2(\mathbb{R}^N)\}.
\]
It is known that 
\begin{equation}
\label{eq:IBP}
\int_{\mathbb{R}^N} (-\Delta)^\frac{\theta}{2} \phi(x)\cdot\psi(x)\,\di x
=\int_{\mathbb{R}^N}  \phi(x)(-\Delta)^\frac{\theta}{2}\psi(x)\,\di x
\end{equation}
for $\phi,\psi\in H^{\theta/2}(\mathbb{R}^N)$. See \cites{AMY22,LPPS15}.

We construct a test function. 
The key to the proof of our main theorem in the case of $p=p_{\theta,N}$ is to employ a solution to the adjoint heat equation {\rm (AHE)} below, as a test function.
Let  $\delta>0$ be  such that $\delta^\theta\in (0,1/64)$ and $f_\delta:[0,\infty)\to[0,\infty)$ be such that $f_\delta\in C^\infty_0(\mathbb{R})$, $f_\delta$ is nonincreasing on $(4\delta^\theta,\infty)$, and
\[
f_\delta(\tau) \left\{
\begin{array}{ll}
\le \tau^{-1} 		& \mbox{if} \quad\tau\in (2\delta^\theta, 1/2),\vspace{3pt}\\
= \tau^{-1}	& \mbox{if} \quad\tau\in (4\delta^\theta, 1/4),\vspace{3pt}\\
=0				& \mbox{otherwise}.\vspace{3pt}\\
\end{array}
\right.
\]
Set $c_\delta := [-\log32\delta^\theta]^{-1}$ and $B:=B(0,1)$.
Let $\varphi_\delta$
be a solution to 
\begin{equation*}
{\rm (AHE)} \qquad 
\left\{
	\begin{aligned}
		&-\partial_t \varphi_{\delta}+  (-\Delta)^\frac{\theta}{2}  \varphi_{\delta} = c_\delta f_\delta(|x|^\theta +t ),	\quad && x\in B,\,\, t\in (0,1),\\
		&\varphi_{\delta}(x,t) =  0,							\quad && x\in B^c,\,\, t\in (0,1),\\
		&\varphi_{\delta}(x,1) =  0,									\quad && x\in B.
	\end{aligned}
	\right.
\end{equation*}
If $\theta =2$, $B^c$ in {\rm (AHE)} is replaced by $\partial B$.
Compare with \cite{HIT20}.
Since $f_\delta$ is a smooth function in $B\times(0,1)$, $\varphi_\delta$ satisfies problem {\rm (AHE)} in the classical sense. See {\it e.g.} \cite{LPPS15}.

We collect properties of the fundamental solution $\Gamma_\theta$ and the Dirichlet heat kernel $G_B$ on $B$ to reveal the properties of $\varphi_{\delta}$.
The fundamental solution $\Gamma_\theta$ is a positive and smooth function in $\mathbb{R}^N\times(0,\infty)$ and  has the following properties, 
\begin{eqnarray}
\label{eq:F1}
 & & \Gamma_\theta(x,t)=t^{-\frac{N}{\theta}}\Gamma_\theta\left(t^{-\frac{1}{\theta}}x,1\right),\\
\label{eq:F2}
 & & C^{-1}(1+|x|)^{-N-\theta}\le \Gamma_\theta(x,1)\le C(1+|x|)^{-N-\theta},\\
\label{eq:F3}
 & &  \mbox{$\Gamma_\theta(\cdot,1)$ is radially symmetric and $\Gamma_\theta(x,1)\le \Gamma_\theta(y,1)$ if $|x|\ge |y|$},
\end{eqnarray}
for all $x$, $y\in\mathbb{R}^N$ and $0<s<t$
(see {\it e.g.}, \cites{BSV17,HI18}). 
The Dirichlet heat kernel $G_B$ is continuous in $\mathbb{R}^N\times\mathbb{R}^N\times(0,\infty)$
and smooth in  ${B}\times {B}\times(0,\infty)$.
No explicit formulas of $G_B$ can be expected  even in the case of $\theta=2$ but
the following two-sided estimate  of $G_B$ was obtained in \cites{CKS10, Cho06, R05, Z02}:

\begin{lemma}
\label{Lemma:TSE}
Let $N\ge1$ and $\theta\in(0,2]$.
For $x\in \overline{B}$, set $\dist_B(x) := \Dist(x,\partial B) = 1-|x|$.
There exist $C_1, C_2,c_1, c_2>0$  depending only on $N$ such that
\begin{equation}
\label{eq:TSE}
\begin{split}
&C_1\left(1\wedge \frac{\dist_B(x)^\frac{\theta}{2}}{\sqrt{t}}\right)\left(1\wedge \frac{\dist_B(y)^\frac{\theta}{2}}{\sqrt{t}}\right)\Gamma_\theta(x-y,c_1t)\\
&\qquad\qquad \le
G_B(x,y,t) \le C_2\left(1\wedge \frac{\dist_B(x)^\frac{\theta}{2}}{\sqrt{t}}\right)\left(1\wedge \frac{\dist_B(y)^\frac{\theta}{2}}{\sqrt{t}}\right)\Gamma_\theta(x-y,c_2t)
\end{split}
\end{equation}
for all $x,y\in\overline{B}$ and $t\in(0,1]$.
Here,  $a\wedge b := \min\{a,b\}$ for $a,b\in\mathbb{R}$.
\end{lemma}

Define $\psi_\delta (x,t) := \varphi_\delta(x,1-t)$.  Then,
\begin{equation}
\label{eq:Duhamel}
\begin{split}
\psi_\delta(x,t) = c_\delta\int_0^t \int_BG_B(x,y,t-s) f_\delta(|y|^\theta + 1-s) \,\di y\di s
\end{split}
\end{equation}
for all $x\in B$ and $t\in (0,1)$.
Using the properties of $\Gamma_\theta$ and $G_B$ and \eqref{eq:Duhamel},
we obtain properties of $\psi_\delta$ ({\it i.e.} $\varphi_\delta$).

\begin{lemma}
Let $\delta>0$ be such that $\delta^{\theta}\in(0,1/64)$.
There exists a constant $c>0$ independent of $\delta$ such that
\begin{equation}
\label{eq:LEAHE}
\varphi_\delta(x, 0) \ge c \quad \mbox{for all} \quad x\in B(0,\delta).
\end{equation}
\end{lemma}
\begin{proof}
We shall calculate  $\psi_\delta(x,1) (= \varphi_\delta(x,0))$. By \eqref{eq:TSE} and \eqref{eq:Duhamel} we have
\begin{equation}
\label{eq:2.3}
\begin{split}
\psi_\delta(x,1) 
&= c_\delta\int_0^1 \int_BG_B(x,y,1-s) f_\delta(|y|^\theta + 1-s) \,\di y\di s\\
&\ge C_1  c_\delta\int_0^1 \int_B \left(1\wedge \frac{\dist_B(x)^\frac{\theta}{2}}{\sqrt{1-s}}\right)\left(1\wedge \frac{\dist_B(y)^\frac{\theta}{2}}{\sqrt{1-s}}\right)\\
&\qquad\qquad\qquad\qquad\qquad\qquad \times \Gamma_\theta(x-y,c_1(1-s)) f_\delta(|y|^\theta+1-s) \, \di y \di s\\
&\ge C_1 c_\delta \int_{\delta^\theta}^{1-4\delta^\theta} \int_{B(0, 2^{-1}(1-s)^\frac{1}{\theta})} \left(1\wedge \frac{\dist_B(x)^\frac{\theta}{2}}{\sqrt{1-s}}\right)\left(1\wedge \frac{\dist_B(y)^\frac{\theta}{2}}{\sqrt{1-s}}\right)\\
&\qquad\qquad\qquad\qquad\qquad\qquad \times \Gamma_\theta(x-y,c_1(1-s)) f_\delta(|y|^\theta+1-s) \, \di y \di s\\
\end{split}
\end{equation}
for $x\in B$.
Since $x\in B(0,\delta)$, $y\in B(0,2^{-1}(1-s)^{1/\theta})$, and $s\in(\delta^\theta, 1-4\delta^\theta)$, we have
\begin{equation*}
\dist_B(x) =1-|x|> 1-\delta = \frac{1-\delta}{(1-\delta^\theta)^\frac{1}{\theta}}(1-\delta^\theta)^\frac{1}{\theta}\ge C(1-\delta^\theta)^\frac{1}{\theta} \ge  C(1-s)^\frac{1}{\theta}, 
\end{equation*}
\begin{equation*}
\dist_B(y)= 1-|y| \ge 1- \frac{1}{2}(1-s)^\frac{1}{\theta}\ge \frac{1}{2} (1-s)^\frac{1}{\theta},
\end{equation*}
and
\[
|x-y| \le |x| + |y| <\delta + \frac{1}{2}(1-s)^\frac{1}{\theta} \le C(1-s)^\frac{1}{\theta}.
\]
These together with \eqref{eq:F1}, \eqref{eq:F2}, and \eqref{eq:2.3} imply that
\begin{equation*}
\begin{split}
\psi_\delta(x,1) 
&\ge C  c_\delta\int_{\delta^\theta}^{1-4\delta^\theta} \int_{B(0, 2^{-1}(1-s)^\frac{1}{\theta})} \Gamma_\theta(x-y,c_1(1-s)) f_\delta(|y|^\theta+1-s) \, \di y \di s\\
&\ge C c_\delta \int_{\delta^\theta}^{1-4\delta^\theta}\int_{B(0, 2^{-1}(1-s)^\frac{1}{\theta})} (1-s)^{-\frac{N}{\theta}} f_\delta(|y|^\theta+1-s) \, \di y \di s.\\
\end{split}
\end{equation*}
Since $s<1-4\delta^\theta$, we have $4\delta^\theta < |y|^\theta+1-s$ for $y\in B(0, 2^{-1}(1-s)^{1/\theta})$.
This together with $|y| < (1-s)^{1/\theta}$ implies that 
\begin{equation*}
\begin{split}
\psi_\delta(x,1) 
&\ge Cc_\delta  \int_{\delta^\theta}^{1-4\delta^\theta}\int_{B(0, 2^{-1}(1-s)^\frac{1}{\theta})} (1-s)^{-\frac{N}{\theta}} f_\delta(2(1-s)) \, \di y \di s\\
&\ge C  c_\delta\int_{\delta^\theta}^{1-4\delta^\theta}f_\delta(2(1-s)) \,  \di s\\
&\ge Cc_\delta  \int_{\{\delta^\theta<s<1-4\delta^\theta\}\cap \{4\delta^\theta < 2(1-s) < 1/4\} }(1-s)^{-1} \,  \di s\\
&= C c_\delta\int_{{7/8}}^{1-4\delta^\theta}(1-s)^{-1} \,  \di s = -Cc_\delta \log32\delta^\theta = C\\
\end{split}
\end{equation*}
for all $x\in B(0,\delta)$. Thus, the proof is complete.
\end{proof}
\section{Proof of the main theorem.}

In this section we prove Theorem~\ref{Theorem:NS}.
\begin{lemma}
\label{Lemma:3.1}
Let $p=p_{\theta,N}$ and  $\delta>0$ be such that $\delta^\theta\in (0,1/64)$.
If \eqref{eq:kihon} holds with $\varphi (x,t) = \varphi_{\delta,\tau}(x,t) := \varphi_\delta(x/\tau^{1/\theta},t/\tau)$ for some $\tau\in (T/2,T)$,
then assertion~{\rm (ii)} in Theorem~{\rm \ref{Theorem:NS}} holds.
\end{lemma}
\begin{proof}
Let $\tau\in(T/2,T)$.
Note that $\supp\, \varphi_{\delta,\tau} \subset \overline{B(0,\tau^{1/\theta})}\times[0,\tau]$.
By the assumption of Lemma~\ref{Lemma:3.1}, we have
\begin{equation}
\label{eq:3.1}
\begin{split}
&\int_{\mathbb{R}^N} \varphi_{\delta,\tau}\left(x,0\right)^\frac{p}{p-1}\, \di\mu(x) \\
&\qquad\qquad\le C \int_0^\tau\int_{B(0,\tau^\frac{1}{\theta})} \left|(-\partial_t + (-\Delta)^\frac{\theta}{2}) \varphi_{\delta,\tau}\left(x,t\right)\right|^\frac{p}{p-1} \, \di x\di t.
\end{split}
\end{equation}
It follows from \eqref{eq:LEAHE} that 
\[
\varphi_{\delta,\tau}\left(x,0\right) = \varphi_{\delta}\left(\frac{x}{\tau^\frac{1}{\theta}},0\right) \ge c\quad \mbox{for all} \quad x\in B(0, \delta \tau^\frac{1}{\theta}).
\]
This implies that 
\begin{equation}
\label{eq:3.2}
\begin{split}
\int_{\mathbb{R}^N} \varphi_{\delta,\tau}\left(x,0\right)^\frac{p}{p-1}\, \di\mu(x) 
&\ge \int_{B(0, \delta \tau^\frac{1}{\theta})} \varphi_{\delta,\tau}\left(x,0\right)^\frac{p}{p-1}\, \di\mu(x)\\
& \ge c^\frac{p}{p-1} \mu (B(0, \delta \tau^\frac{1}{\theta})). 
\end{split}
\end{equation}
On the other hand, by the definition of $\varphi_{\delta}$, we see that
\begin{equation*}
\begin{split}
&\int_0^\tau\int_{B(0,\tau^\frac{1}{\theta})} \left|(-\partial_t + (-\Delta)^\frac{\theta}{2}) \varphi_{\delta,\tau}\left(x,t\right)\right|^\frac{p}{p-1} \, \di x\di t\\
&= \tau^{-\frac{p}{p-1}}\int_0^\tau\int_{B(0,\tau^\frac{1}{\theta})} \left|[(-\partial_t + (-\Delta)^\frac{\theta}{2}) \varphi_{\delta}]\left(\frac{x}{\tau^\frac{1}{\theta}}, \frac{t}{\tau}\right)\right|^\frac{p}{p-1} \, \di x\di t\\
& =c_\delta^\frac{p}{p-1}\tau^{-\frac{p}{p-1}} \int_0^\tau\int_{B(0,\tau^\frac{1}{\theta})} f_\delta\left(\frac{|x|^\theta}{\tau}+ \frac{t}{\tau}\right)^\frac{p}{p-1} \, \di x\di t\\
&\le Cc_\delta^\frac{p}{p-1}\tau^{\frac{N}{\theta}-\frac{1}{p-1}} \int_0^1\int_{B(0,1)} f_\delta\left(|x|^\theta+ t\right)^\frac{p}{p-1} \, \di x\di t\\
&= Cc_\delta^\frac{p}{p-1} \int_0^1\int_{B(0,1)} f_\delta\left(|x|^\theta+ t\right)^\frac{p}{p-1} \, \di x\di t,
\end{split}
\end{equation*}
where we used $p = p_{\theta, N}$.
We shall calculate the integral in the above inequality.
By the definition of $f_\delta$ we have
\begin{equation*}
\begin{split}
& \int_0^1\int_{B(0,1)} f_\delta\left(|x|^\theta+ t\right)^\frac{p}{p-1} \, \di x\di t\\
&\le \int\int_{\{2\delta^\theta\le |x|^\theta+t \le 1/2\}} (|x|^\theta+ t)^{-\frac{p}{p-1}} \, \di x\di t\\
&= 2 \int\int_{\{2\delta^\theta\le |x|^\theta+s^2 \le 1/2\}} (|x|^\theta+ s^2)^{-\frac{p}{p-1}} s\, \di x\di s\\
&= C \int\int_{\{2\delta^\theta\le r_1^\theta+s^2 \le 1/2, r_1>0\}} (r_1^\theta+ s^2)^{-\frac{p}{p-1}} sr_1^{N-1}\, \di r_1\di s\\
&= \frac{2C}{\theta} \int\int_{\{2\delta^\theta\le r_2^2+s^2 \le 1/2, r_2>0\}} (r_2^2+ s^2)^{-\frac{p}{p-1}} sr_2^{\frac{2N}{\theta}-1}\, \di r_2\di s\\
&\le C\int\int_{\{2\delta^\theta\le r_2^2+s^2 \le 1/2\}} (\sqrt{r_2^2+ s^2})^{\frac{2N}{\theta}-\frac{2p}{p-1}} \, \di r_2\di s = C\int_{\sqrt{2\delta^\theta}}^{\sqrt{1/2}} r^{-1} \, \di r \le C c_\delta^{-1},
\end{split}
\end{equation*}
where in the third line we used the change of variables $s=\sqrt{t}$,
in the fourth line we used the change of variables $r_1=|x|$,
in the fifth line we used the change of variables $r_2 = r_1^{\theta/2}$,
and 
in the last line we used the change of variables $r= \sqrt{r_2^2+s^2}$.
Summarizing the above two calculations, we obtain
\begin{equation}
\label{eq:3.3}
\begin{split}
&\int_0^\tau\int_{B(0,\tau^\frac{1}{\theta})} \left|(-\partial_t + (-\Delta)^\frac{\theta}{2}) \varphi_{\delta,\tau}\left(x,t\right)\right|^\frac{p}{p-1} \, \di x\di t \\
&\le C c_\delta^{\frac{1}{p-1}} 
\le C \left[\log\left(e+ \frac{1}{\delta^\theta}\right)\right]^{-\frac{1}{p-1}}.
\end{split}
\end{equation}
%
Combining \eqref{eq:3.1}, \eqref{eq:3.2}, and \eqref{eq:3.3}, we obtain
\[
\mu (B(0, \delta \tau^\frac{1}{\theta})) \le C \left[\log\left(e+ \frac{1}{\delta^\theta}\right)\right]^{-\frac{1}{p-1}}.
\]
 By a translation, we see that  
 \[
\mu (B(z, \delta \tau^\frac{1}{\theta})) \le C \left[\log\left(e+ \frac{1}{\delta^\theta}\right)\right]^{-\frac{1}{p-1}}
\]
for all $z\in\mathbb{R}^N$ and $\delta>0$ such that $\delta^\theta\in (0,1/64)$.

Hereafter, we verify that the above inequality holds  for all 
$\delta\in (0,1)$ with $\tau=T$.
For this purpose, an appropriate covering lemma is used (see {\it e.g.}~\cite[Lemma~2.2]{BH24}).
Fix $\delta_1>0$ so that $\delta_1^\theta = 64 \delta^\theta$. Note that $\delta_1\in(0,1)$.
We consider the following family of balls $\{B(z',\delta \tau^{1/\theta}/5 ); z' \in B(z,\delta _1T^{1/\theta} )\}$ which covers $B(z,\delta _1T^{1/\theta} )$.
By Vitali's covering lemma (see {\it e.g.}~\cite[Section~1.5]{EG15}), we can extract a disjoint family of balls $\{B(z_j,\delta \tau^{1/\theta}/5 ); z_j \in B(z,\delta _1T^{1/\theta} ), j\in J\}$
for some countable family of indices $J$ satisfying
\[
B(z, \delta_1 T^\frac{1}{\theta}) \subset \bigcup_{j\in J} B(z_j, \delta \tau^\frac{1}{\theta}).
\]
By the construction, it follows that
\begin{equation*}
\begin{split}
\omega_N \left(\frac{\delta \tau^\frac{1}{\theta}}{5}\right)^N  \sharp J =  \sum_{j\in J} \left|B\left(z_j, \frac{\delta \tau^\frac{1}{\theta}}{5}\right)\right|
&\le \left|B\left(z, \delta_1 T^\frac{1}{\theta}+ \frac{\delta \tau^\frac{1}{\theta}}{5}\right)\right|\\
& = \omega_N \left(\delta_1 T^\frac{1}{\theta} + \frac{\delta \tau^\frac{1}{\theta}}{5}\right)^N,
\end{split}
\end{equation*}
where $\sharp J$ is the cardinal number of $J$ and $\omega_N = |B(0,1)|$.
This implies that 
\begin{equation*}
\begin{split}
\sharp J \le C\left(\frac{\delta \tau^\frac{1}{\theta}}{5}\right)^{-N} \left[ (\delta_1T^\frac{1}{\theta})^N + \left(\frac{\delta \tau^\frac{1}{\theta}}{5}\right)^{N}\right]
&\le C\left[\left(\frac{\delta_1 T^\frac{1}{\theta}}{\delta \tau^\frac{1}{\theta}}\right)^N+1\right] \\
&\le C\left(\frac{\delta_1}{\delta}\right)^N \le C.
\end{split}
\end{equation*}
Here, we used $\tau \in (T/2,T)$.
Therefore, we see that
\begin{equation*}
\begin{split}
& \mu (B(z, \delta_1T^\frac{1}{\theta})) 
\le \sum_{j\in J}  \mu (B(z_j, \delta \tau^\frac{1}{\theta})) \\
&\le C\sharp J  \left[\log\left(e+ \frac{1}{\delta^\theta}\right)\right]^{-\frac{1}{p-1}}
\le C \left[\log\left(e+ \frac{1}{\delta_1^\theta}\right)\right]^{-\frac{1}{p-1}}\\
&= C \left[\log\left(e+ \frac{T}{\delta_1^\theta  T}\right)\right]^{-\frac{N}{\theta}}
\le C \left[\log\left(e+ \frac{T^\frac{1}{\theta}}{\delta_1  T^\frac{1}{\theta}}\right)\right]^{-\frac{N}{\theta}}
\end{split}
\end{equation*}
for all $z\in \mathbb{R}^N$ and $\delta_1\in(0,1)$.
Set $\sigma := \delta_1 T^{1/\theta} \in (0, T^{1/\theta})$.
Since $z\in \mathbb{R}^N$ is arbitrary, we get 
\[
 \sup_{z\in\mathbb{R}^N}\mu (B(z, \sigma)) \le  C \left[\log\left(e+ \frac{T^\frac{1}{\theta}}{\sigma}\right)\right]^{-\frac{N}{\theta}}
\]
for all $\sigma\in (0,T^{1/\theta})$.
Thus, the proof is complete.
\end{proof}
\begin{proof}[Proof of Theorem~{\rm \ref{Theorem:NS}}]
We divide the proof into two cases.

\noindent \underline{{\bf Case:} $p\neq p_{\theta,N}$.}
Let $\sigma\in (0,T^{1/\theta})$ and $\zeta\in C^\infty_0(\mathbb{R}^N)$ be such that
\[
0\le \zeta \le 1\quad\mbox{in}\quad \mathbb{R}^N, \qquad \zeta =1 \quad \mbox{in} \quad B(0,1/2), \qquad \supp\, \zeta \subset B(0,1).
\]
Set $\zeta_\sigma(x) := \zeta(x/\sigma)$ for $x\in \mathbb{R}^N$.
In addition, let $\psi \in C^\infty([0,1])$ be such that 
\[
0\le \psi \le 1\quad\mbox{in}\quad [0,1], \qquad \psi =1 \quad \mbox{in} \quad [1,1/2], \qquad \psi =0 \quad\mbox{in} \quad [3/4,1].
\]
Set $\psi_\sigma (t) := \psi(t/\sigma^\theta)$ for $t\in[0,\infty)$.
Then 
\[
\zeta_\sigma(x) \psi_\sigma (t)  \in C^\infty_0 (\mathbb{R}^N\times[0, 3T/4 ]) \quad \mbox{with} \quad 
\zeta_\sigma(x) \psi_\sigma (3T/4)=0 \quad \mbox{for all}\quad x\in \mathbb{R}^N. 
\]
We can substitute $\varphi(x,t)=\zeta_\sigma(x) \psi_\sigma (t) $ into \eqref{eq:kihon} and obtain
\begin{equation}
\label{eq:C0}
\int_{\mathbb{R}^N} \zeta_\sigma(0)^\frac{p}{p-1}\, \di\mu \le C \int_0^{\sigma^\theta}\int_{B(0,\sigma)} 
[|\partial_t \psi_\sigma|^\frac{p}{p-1} + |(-\Delta)^\frac{\theta}{2} \zeta_\sigma|^\frac{p}{p-1}] \, \di x\di t.
\end{equation}
Since $\zeta_\sigma\equiv1$ on $B(0,\sigma/2)$, we have
\begin{equation}
\label{eq:C00}
\int_{\mathbb{R}^N} \zeta_\sigma(0)^\frac{p}{p-1}\, \di\mu  \ge \int_{B(0,\sigma/2)} \zeta_\sigma(0)^\frac{p}{p-1}\, \di\mu   = \mu (B(0,\sigma/2)).
\end{equation}
On the other hand, since
\begin{equation*}
\begin{split}
&|\partial_t \psi_\sigma(t)| = \sigma^{-\theta}|\partial_t\psi(t/\sigma^\theta)| \le C\sigma^{-\theta}
\quad\mbox{and}\quad
(-\Delta)^\frac{\theta}{2} \zeta_\sigma (x) = \sigma^{-\theta}  (-\Delta)^\frac{\theta}{2} \zeta(x/\sigma)
\end{split}
\end{equation*}
for all $x\in B(0,\sigma)$ and $t\in [0,\sigma^\theta)$, and $\zeta\in C^\infty_0(\mathbb{R}^N)\subset H^{\theta, p/(p-1)}(\mathbb{R}^N)$ (see {\it e.g.}~\cite[Theorem~7.38]{Adams75}),  we have
\begin{equation}
\label{eq:C000}
\begin{split}
&\int_0^{\sigma^\theta}\int_{B(0,\sigma)} 
[|\partial_t \psi_\sigma|^\frac{p}{p-1} + |(-\Delta)^\frac{\theta}{2} \zeta_\sigma|^\frac{p}{p-1}] \, \di x\di t\\
&\le C\sigma^{N-\frac{\theta }{p-1}} + \sigma^{-\frac{\theta}{p-1}}\int_{B(0,\sigma)} |(-\Delta)^\frac{\theta}{2} \zeta(x/\sigma)|^\frac{p}{p-1} \, \di x\\
&\le C\sigma^{N-\frac{\theta }{p-1}} +C \sigma^{N-\frac{\theta}{p-1}}\int_{B(0,1)} |(-\Delta)^\frac{\theta}{2} \zeta(x)|^\frac{p}{p-1} \, \di x \le C \sigma^{N-\frac{\theta}{p-1}}\\
\end{split}
\end{equation}
for all $\sigma\in (0,T^{1/\theta})$.
Combining \eqref{eq:C0}, \eqref{eq:C00}, and \eqref{eq:C000}, we get
\[
\mu (B(0,\sigma/2)) \le C \sigma^{N-\frac{\theta}{p-1}}
\]
for all $\sigma\in (0,T^{1/\theta})$.  By a translation,   
\[
\mu (B(z,\sigma/2)) \le C \sigma^{N-\frac{\theta}{p-1}}
\]
for all $z\in \mathbb{R}^N$ and all $\sigma\in (0,T^{1/\theta})$.
By the similar argument to the proof of Lemma~\ref{Lemma:3.1}, we obtain
\[
\sup_{z\in\mathbb{R}^N}\mu (B(z,\sigma)) \le C \sigma^{N-\frac{\theta}{p-1}}
\]
for all $\sigma\in (0,T^{1/\theta})$. Thus, the proof is complete.

\vspace{5pt}
\noindent \underline{{\bf Case:} $p= p_{\theta,N}$.}
Fix $\tau\in(T/2,T)$.
Let  $\delta>0$ be such that $\delta^\theta\in(0,1/64)$.
By virtue of Lemma~\ref{Lemma:3.1},
it suffices to verify that \eqref{eq:kihon} holds with $\varphi=\varphi_{\delta, \tau}$.
Note that $\varphi_{\delta, \tau} \not\in C^\infty_0(\mathbb{R}^N\times[0,\tau])$,
therefore, we use an approximation.
Let $\eta\in C^\infty_0(\mathbb{R}^N)$ be such that $\eta$ is radially symmetric,
\[
\eta\ge0 \quad\mbox{in} \quad \mathbb{R}^N, \qquad \supp\, \eta \subset \overline{B(0,1)}, \qquad \int_{\mathbb{R}^N} \eta(x) \, \di x = 1.
\]
For any $\epsilon >0$, set  
\[
\eta_{\epsilon} (x) := (\epsilon \tau^\frac{1}{\theta})^{-N} \eta\left(\frac{x}{\epsilon \tau^\frac{1}{\theta}}\right), \qquad x\in\mathbb{R}^N.
\]
We have 
\begin{align}
	&\notag \varphi_{\delta,\tau} * \eta_\epsilon \in C^\infty_0 (\mathbb{R}^N\times[0,\tau]),\\
	&\notag \supp\, (\varphi_{\delta,\tau} * \eta_\epsilon)  \subset \overline{B(0, (1+\epsilon)\tau^\frac{1}{\theta})} \times[0,\tau],\\	
	&\notag [\varphi_{\delta,\tau} * \eta_\epsilon](x,t) \ge0 \quad \mbox{for all}\quad  x\in \mathbb{R}^N, \, \,t\in[0,\infty),\\
	& \notag [\varphi_{\delta,\tau} * \eta_\epsilon](x,\tau) =0 \quad \mbox{for all}\quad  x\in \mathbb{R}^N.
\end{align}
Then we can substitute $\varphi = \varphi_{\delta,\tau} * \eta_\epsilon$ into \eqref{eq:2.1} and get
\begin{equation}
\label{eq:C1}
\begin{split}
&\int_0^\tau \int_{\mathbb{R}^N} u^p (\varphi_{\delta,\tau} * \eta_\epsilon)^\frac{p}{p-1} \,\di x\di t + \int_{\mathbb{R}^N}  (\varphi_{\delta,\tau} * \eta_\epsilon)(0)^\frac{p}{p-1}\, \di\mu \\
&\le \frac{p}{p-1}\int_0^\tau\int_{B(0, (1+\epsilon)\tau^\frac{1}{\theta})} u  (\varphi_{\delta,\tau} * \eta_\epsilon)^\frac{1}{p-1} (-\partial_t + (-\Delta)^\frac{\theta}{2}) (\varphi_{\delta,\tau} * \eta_\epsilon)\,\di x\di t\\
\end{split}
\end{equation}
Since $\mu$ is a Radon measure on $\mathbb{R}^N$ and  $u^p\in L^1_{\rm loc} (\mathbb{R}^N\times[0,T))$, it is easy to see  that
\begin{equation*}
\begin{split}
& \lim_{\epsilon\to0^+} \int_0^\tau \int_{\mathbb{R}^N} u^p (\varphi_{\delta,\tau} * \eta_\epsilon)^\frac{p}{p-1} \,\di x\di t = \int_0^\tau \int_{\mathbb{R}^N} u^p \varphi_{\delta,\tau}^\frac{p}{p-1} \,\di x\di t, \\
&\lim_{\epsilon\to0^+} \int_{\mathbb{R}^N} (\varphi_{\delta,\tau} * \eta_\epsilon)(0)^\frac{p}{p-1}\, \di\mu = 
\int_{\mathbb{R}^N} \varphi_{\delta,\tau} (0)^\frac{p}{p-1}\, \di\mu.
\end{split}
\end{equation*}

We consider the right hand side of \eqref{eq:C1}.
Since $f_\delta$ is a smooth function, for $x\in B(0, (1-\epsilon)\tau^{1/\theta})$ and $t\in(0,\tau)$
we have
\begin{equation*}
\begin{split}
&(-\partial_t + (-\Delta)^\frac{\theta}{2})(\varphi_{\delta,\tau} * \eta_\epsilon)(x,t) \\
&= 
\int_{B(0, \epsilon \tau^\frac{1}{\theta})} (-\partial_t + (-\Delta)^\frac{\theta}{2}_x)\varphi_{\delta,\tau} (x-y,t)\cdot \eta_\epsilon(y)\, \di y\\
&= \int_{B(0, \epsilon \tau^\frac{1}{\theta})} \tau c_\delta f\left(\frac{|x-y|^\theta}{\tau}+\frac{t}{\tau}\right) \eta_\epsilon(y)\, \di y\\
&\to \tau c_\delta f\left(\frac{|x|^\theta}{\tau}+\frac{t}{\tau}\right)  = (-\partial_t + (-\Delta)^\frac{\theta}{2})\varphi_{\delta,\tau} (x,t) \quad \mbox{as} \quad \epsilon \to 0^+.
\end{split}
\end{equation*}
Since $u\in L^1_{\rm loc} (\mathbb{R}^N\times[0,T))$, this implies that
\begin{equation*}
\begin{split}
&\lim_{\epsilon\to0^+} \int_0^\tau\int_{B(0, (1-\epsilon)\tau^\frac{1}{\theta})} u  (\varphi_{\delta,\tau} * \eta_\epsilon)^\frac{1}{p-1} (-\partial_t + (-\Delta)^\frac{\theta}{2}) (\varphi_{\delta,\tau} * \eta_\epsilon)\,\di x\di t\\
&=\int_0^\tau\int_{B(0, \tau^\frac{1}{\theta})} u  \varphi_{\delta,\tau}^\frac{1}{p-1} (-\partial_t + (-\Delta)^\frac{\theta}{2}) \varphi_{\delta,\tau} \,\di x\di t.
\end{split}
\end{equation*}
We shall check
\begin{equation}
\label{eq:C2}
(-\Delta)^\frac{\theta}{2}_x \eta_\epsilon(x-y) = (-\Delta)^\frac{\theta}{2}_y \eta_\epsilon(y-x) 
\end{equation}
for all $x,y\in\mathbb{R}^N$.
Since $\eta_\epsilon$ is radially symmetric, by the definition of the fractional Laplacian $(-\Delta)^{\theta/2}$, 
\begin{equation*}
\begin{split}
(-\Delta)^\frac{\theta}{2}_x \eta_\epsilon(x-y)  
&= A(N,\theta) P.V.\int_{\mathbb{R}^N} \frac{\eta_\epsilon(x-y)-\eta_\epsilon(x-y-z)}{|x-y-z|^{N+\theta}} \, \di z\\
&= A(N,\theta) P.V.\int_{\mathbb{R}^N} \frac{\eta_\epsilon(x-y)-\eta_\epsilon(x-y+z')}{|x-y+z'|^{N+\theta}} \, \di z'\\
&= A(N,\theta) P.V.\int_{\mathbb{R}^N} \frac{\eta_\epsilon(y-x)-\eta_\epsilon(y-x-z')}{|y-x-z'|^{N+\theta}} \, \di z'\\
&=(-\Delta)^\frac{\theta}{2}_y \eta_\epsilon(y-x),
\end{split}
\end{equation*}
where in the second line we used the change of variables $z'=-z$.
Then \eqref{eq:C2} follows.
Next, we shall check
\begin{equation}
\label{eq:C3}
(-\Delta)^\frac{\theta}{2} \varphi_{\delta,\tau} (x,t) < 0 \quad \mbox{for all} \quad x\in\mathbb{R}^N\setminus B(0,\tau^\frac{1}{\theta}), \,\, t\in(0,\tau).
\end{equation}
Since 
\begin{equation*}
\begin{split}
&\varphi_{\delta,\tau}(x,t)=0 \quad\mbox{for all} \quad x\in\mathbb{R}^N\setminus B(0,\tau^\frac{1}{\theta}),\,\,t\in(0,\tau),\\
&\varphi_{\delta,\tau}(x,t)>0 \quad\mbox{for all}\quad x\in B(0,\tau^\frac{1}{\theta}),\,\, t\in(0,\tau),
\end{split}
\end{equation*}
 by the definition of the fractional Laplacian $(-\Delta)^{\theta/2}$, 
\begin{equation*}
\begin{split}
(-\Delta)^\frac{\theta}{2} \varphi_{\delta,\tau}(x,t)
&= A(N,\theta) P.V.\int_{\mathbb{R}^N} \frac{\varphi_{\delta,\tau} (x,t)-\varphi_{\delta,\tau}(z,t)}{|x-z|^{N+\theta}} \, \di z\\
&= -A(N,\theta) P.V.\int_{\mathbb{R}^N} \frac{\varphi_{\delta,\tau} (z,t)}{|x-z|^{N+\theta}} \, \di z<0.
\end{split}
\end{equation*}

Let $E_\epsilon := B(0, (1+\epsilon) \tau^{1/\theta})\setminus B(0,(1-\epsilon)\tau^{1/\theta})$.
It follows from  \eqref{eq:IBP}, \eqref{eq:C2}, and \eqref{eq:C3} that
for $x\in E_\epsilon$ and  $t\in(0,\tau)$,
\begin{equation}
\label{eq:C5}
\begin{split}
&(-\partial_t + (-\Delta)^\frac{\theta}{2}_x) (\varphi_{\delta,\tau} * \eta_\epsilon)(x,t)\\
&= \int_{\mathbb{R}^N} (-\partial_t + (-\Delta)^\frac{\theta}{2}_x) \eta_\epsilon (x-y) \cdot \varphi_{\delta,\tau}(y,t) \, \di y\\
&= \int_{\mathbb{R}^N} (-\partial_t + (-\Delta)^\frac{\theta}{2}_y) \eta_\epsilon (y-x) \cdot \varphi_{\delta,\tau}(y,t) \, \di y\\
&= \int_{B(0,\tau^\frac{1}{\theta})}  \eta_\epsilon (y-x)(-\partial_t + (-\Delta)^\frac{\theta}{2}_y)\varphi_{\delta,\tau}(y,t) \, \di y\\
&\qquad\qquad\qquad\qquad+ \int_{\mathbb{R}^N\setminus B(0,\tau^\frac{1}{\theta})}\eta_\epsilon (y-x) (-\Delta)^\frac{\theta}{2}_y\varphi_{\delta,\tau}(y,t) \, \di y
\\
&<  \int_{B(0,\tau^\frac{1}{\theta})}  \eta_\epsilon (y-x)(-\partial_t + (-\Delta)^\frac{\theta}{2}_y)\varphi_{\delta,\tau}(y,t) \, \di y\\
&=c_\delta \tau^{-1}\int_{B(0,\tau^\frac{1}{\theta})}  \eta_\epsilon (y-x)f\left(\frac{|y|^\theta}{\tau} + \frac{t}{\tau}\right)\, \di y\\
&=c_\delta \tau^{-1}\int_{\mathbb{R}^N}  \eta_\epsilon (y-x)f\left(\frac{|y|^\theta}{\tau} + \frac{t}{\tau}\right)\, \di y.
\end{split}
\end{equation}
Since $u\in L^1_{\rm loc} (\mathbb{R}^N\times[0,T))$, this implies that
\begin{equation*}
\begin{split}
&\limsup_{\epsilon\to0^+} \int_0^\tau\int_{E_\epsilon} u  (\varphi_{\delta,\tau} * \eta_\epsilon)^\frac{1}{p-1} (-\partial_t + (-\Delta)^\frac{\theta}{2}) (\varphi_{\delta,\tau} * \eta_\epsilon)\,\di x\di t\\
&\le c_\delta \tau^{-1}\limsup_{\epsilon\to0^+} \int_0^\tau\int_{E_\epsilon} u  (\varphi_{\delta,\tau} * \eta_\epsilon)^\frac{1}{p-1}\int_{\mathbb{R}^N}  \eta_\epsilon (x-y)f\left(\frac{|y|^\theta}{\tau} + \frac{t}{\tau}\right)\, \di y \di x\di t=0.\\
\end{split}
\end{equation*}

Now letting $\epsilon\to0^+$ in \eqref{eq:C1}, we obtain 
\begin{equation*}
\begin{split}
&\int_0^\tau \int_{\mathbb{R}^N} u^p \varphi_{\delta,\tau}^\frac{p}{p-1} \,\di x\di t + \int_{\mathbb{R}^N}  \varphi_{\delta,\tau} (0)^\frac{p}{p-1}\, \di\mu \\
&\qquad\qquad\le \frac{p}{p-1}\int_0^\tau\int_{B(0, \tau^\frac{1}{\theta})} u  \varphi_{\delta,\tau}^\frac{1}{p-1} (-\partial_t + (-\Delta)^\frac{\theta}{2}) \varphi_{\delta,\tau}\,\di x\di t.\\
\end{split}
\end{equation*}
By the same argument as in the proof of Lemma~\ref{Lemma:kihon}, we see that \eqref{eq:kihon} holds with $\varphi=\varphi_{\delta, \tau}$. Thus, the proof is complete.
\end{proof}

So far we have focused on the case of $\theta\in(0,2)$.
At the end of this section we give brief comments on the differences between the proofs in the cases of $\theta\in (0,2)$ and  $\theta=2$.

\begin{proof}[Sketch of the proof in the case of $\theta=2$]
The major difference in the proofs appears in \eqref{eq:C5}.
It follows from Hopf's lemma that
\[
\nabla \varphi_{\delta,\tau}(y,t)\cdot n(y) <0 \quad\mbox{for all}\quad y\in\partial B(0,\tau^\frac{1}{2}),\,\, t\in (0,\tau),
\]
where $n(y)$ is the outer normal unit vector at $y\in \partial B(0,\tau^{1/2})$.
Let $x\in E_\epsilon =  B(0, (1+\epsilon) \tau^{1/2})\setminus B(0,(1-\epsilon)\tau^{1/2})$.
By Green's identity, we have
\begin{equation*}
\begin{split}
&(-\partial_t -\Delta_x) (\varphi_{\delta,\tau} * \eta_\epsilon)(x,t)\\
&= \int_{B(0,\tau^\frac{1}{2})} (-\partial_t -\Delta_x) \eta_\epsilon (x-y) \cdot \varphi_{\delta,\tau}(y,t) \, \di y\\
&= \int_{B(0,\tau^\frac{1}{2})} (-\partial_t -\Delta_y) \eta_\epsilon (y-x) \cdot \varphi_{\delta,\tau}(y,t) \, \di y\\
&= \int_{B(0,\tau^\frac{1}{2})}  \eta_\epsilon (y-x)(-\partial_t -\Delta_y)\varphi_{\delta,\tau}(y,t) \, \di y\\
&\qquad\qquad\qquad\qquad+ \int_{\partial B(0,\tau^\frac{1}{2})}\eta_\epsilon (y-x) \nabla \varphi_{\delta,\tau}(y,t)\cdot n(y) \, \di S(y)
\\
&<  \int_{B(0,\tau^\frac{1}{2})}  \eta_\epsilon (y-x)(-\partial_t-\Delta_y)\varphi_{\delta,\tau}(y,t) \, \di y\\
&=c_\delta \tau^{-1}\int_{B(0,\tau^\frac{1}{2})}  \eta_\epsilon (y-x)f\left(\frac{|y|^2}{\tau} + \frac{t}{\tau}\right)\, \di y\\
&=c_\delta \tau^{-1}\int_{\mathbb{R}^N}  \eta_\epsilon (y-x)f\left(\frac{|y|^2}{\tau} + \frac{t}{\tau}\right)\, \di y.
\end{split}
\end{equation*}
The same estimate follows
(compare with \eqref{eq:C5}).
A similar argument then allows us to obtain the desired result.
\end{proof}

\section{Appendix.}
In this section, we prove Proposition~\ref{Proposition:IisQ}.

\begin{proof}[Proof of Proposition~{\rm \ref{Proposition:IisQ}}]
Let $T\in (0,\infty)$ and $\tau\in (0,T)$.
We find $t\in(T-\epsilon,T)$ such that
\begin{equation*}
\begin{split}
\infty>u(x,t) &\ge \int_{0}^{T-2\epsilon} \int_{\mathbb{R}^N} \Gamma_\theta(x-y,t-s) f(u(y,s)) \,\di y \di s\\
&= \int_0^{T-2\epsilon} \int_{\mathbb{R}^N}  (t-s)^{-\frac{N}{\theta}} 
\Gamma_\theta\left(\frac{x-y}{(t-s)^\frac{1}{\theta}},1\right) u(y,s)^p\, \di y\di s\\
&\ge  T^{-\frac{N}{\theta}} \int_0^{T-2\epsilon} \int_{\mathbb{R}^N}  \Gamma_\theta\left(\frac{x-y}{\epsilon^\frac{1}{\theta}},1\right) u(y,s)^p\,\di y \di s
\end{split}
\end{equation*}
for a.a.~$x\in{\mathbb{R}^N}$,
where we used \eqref{eq:F1} and \eqref{eq:F3}.
Since $\epsilon\in(0,T/2)$ is arbitrary, we see that $u^p \in L^1_{\rm loc} (\mathbb{R}^N\times [0,T))$.

For any $\tau\in(0,T)$, let $\varphi\in C^\infty_0(\mathbb{R}^N \times[0,\tau])$ with $\varphi(\tau)=0$. 
It is well-known that
\[
\int_{\mathbb{R}^N} (-\Delta)^\frac{\theta}{2} \Gamma_\theta(x,t) \cdot \varphi(x,t) \, \di x 
= \int_{\mathbb{R}^N}  \Gamma_\theta(x,t) (-\Delta)^\frac{\theta}{2} \varphi(x,t) \, \di x 
\]
for $t\in (0,\tau)$.
This together with the integral by parts implies that
\begin{equation*}
\begin{split}
&\int_{\mathbb{R}^N} \varphi(y,0) \, \di \mu(y)\\
&= \int_{\mathbb{R}^N} \left(\int_0^{\tau}\int_{\mathbb{R}^N}(\partial_t+(-\Delta)^\frac{\theta}{2}_x)\Gamma_\theta(x-y,t)\cdot \varphi(x,t)\,\di x \di t+ \varphi(y,0)\right) \,\di \mu(y)\\
&= \int_{\mathbb{R}^N} \int_0^{\tau}\int_{\mathbb{R}^N}\Gamma_\theta(x-y,t)(-\partial_t+(-\Delta)^\frac{\theta}{2})\varphi(x,t)  \,\di x \di t \di \mu(y)\\
&= \int_0^{\tau_2}\int_{\mathbb{R}^N}\left(\int_{\mathbb{R}^N}\Gamma_\theta(x-y,t)\,\di \mu(y)\right)(-\partial_t+(-\Delta)^\frac{\theta}{2})\varphi(x,t) \,\di x \di t.
\end{split}
\end{equation*}
Similarly, we have
\begin{equation*}
\begin{split}
&\int_0^{\tau}\int_{\mathbb{R}^N} \varphi(y,s)  u(y,s)^p\, \di y \di s\\
&= \int_0^{\tau}\int_{\mathbb{R}^N} \left(\int_s^{\tau}\int_{\mathbb{R}^N}(\partial_t+(-\Delta)_x^\frac{\theta}{2})\Gamma_\theta(x-y,t-s)\cdot \varphi(x,t)\,\di x\di t + \varphi(y,s)\right.\\
&\qquad\times u(y,s)^p\bigg)\,\di y\di s\\
&= \int_0^{\tau}\int_{\mathbb{R}^N} \left(\int_s^{\tau}\int_{\mathbb{R}^N}\Gamma_\theta(x-y,t-s)(-\partial_t+(-\Delta)^\frac{\theta}{2})\varphi(x,t)\,\di x\di t \right)u(y,s)^p\,\di y\di s\\
&= \int_0^{\tau}\int_{\mathbb{R}^N} \left(\int_0^{\tau}\int_{\mathbb{R}^N}\Gamma_\theta(x-y,t-s )u(y,s)^p\,\di y \di s \right)(-\partial_t+(-\Delta)^\frac{\theta}{2})\varphi(x,t)\,\di x \di t.
\end{split}
\end{equation*}
Then
\begin{equation*}
\begin{split}
&\int_{0}^{\tau}\int_{\mathbb{R}^N} u(x,t)(-\partial_t+(-\Delta)^\frac{\theta}{2}) \varphi(x,t)\,\di x\di t\\
&= \int_{0}^{\tau}\int_{\mathbb{R}^N} \biggl(\int_{\mathbb{R}^N}\Gamma_\theta(x-y,t)\, \di \mu(y)\\
&\qquad+\int_0^{\tau}\int_{\mathbb{R}^N}\Gamma_\theta(x-y, t-s)u(y,s)^p\,\di y\di s\biggr)(-\partial_t+(-\Delta)^\frac{\theta}{2})\varphi(x,t)\,\di x\di t\\
&= \int_{\mathbb{R}^N} \varphi(y,0) \, \di \mu(y)+ \int_0^{\tau}\int_{\mathbb{R}^N} \varphi(y,s) u(y,s)^p\, \di y\di s,
\end{split}
\end{equation*}
which implies \eqref{eq:defwsol}. Then Proposition~\ref{Proposition:IisQ} follows.
\end{proof}

	\bibliographystyle{plain}
	\bibliography{ref}

\begin{bibdiv}
\begin{biblist}

\bib{AMY22}{article}{
      author={Abdellaoui, Boumediene},
      author={Mahmoud, Ghoulam Ould~Mohamed},
      author={Youssfi, Ahmed},
       title={Fractional heat equation with singular nonlinearity},
        date={2022},
        ISSN={1662-9981,1662-999X},
     journal={J. Pseudo-Differ. Oper. Appl.},
      volume={13},
      number={4},
       pages={Paper No. 50, 48},
         url={https://doi.org/10.1007/s11868-022-00484-5},
      review={\MR{4488229}},
}

\bib{Adams75}{book}{
      author={Adams, Robert~A.},
       title={Sobolev spaces},
      series={Pure and Applied Mathematics},
   publisher={Academic Press [Harcourt Brace Jovanovich, Publishers], New
  York-London},
        date={1975},
      volume={Vol. 65},
      review={\MR{450957}},
}

\bib{BP85}{article}{
      author={Baras, Pierre},
      author={Pierre, Michel},
       title={Crit\`ere d'existence de solutions positives pour des
  \'{e}quations semi-lin\'{e}aires non monotones},
        date={1985},
        ISSN={0294-1449},
     journal={Ann. Inst. H. Poincar\'{e} Anal. Non Lin\'{e}aire},
      volume={2},
      number={3},
       pages={185\ndash 212},
         url={http://www.numdam.org/item?id=AIHPC_1985__2_3_185_0},
      review={\MR{797270}},
}

\bib{BPSV14}{article}{
      author={Barrios, Bego\~na},
      author={Peral, Ireneo},
      author={Soria, Fernando},
      author={Valdinoci, Enrico},
       title={A {W}idder's type theorem for the heat equation with nonlocal
  diffusion},
        date={2014},
        ISSN={0003-9527,1432-0673},
     journal={Arch. Ration. Mech. Anal.},
      volume={213},
      number={2},
       pages={629\ndash 650},
         url={https://doi.org/10.1007/s00205-014-0733-1},
      review={\MR{3211862}},
}

\bib{BSV17}{article}{
      author={Bonforte, Matteo},
      author={Sire, Yannick},
      author={V\'azquez, Juan~Luis},
       title={Optimal existence and uniqueness theory for the fractional heat
  equation},
        date={2017},
        ISSN={0362-546X,1873-5215},
     journal={Nonlinear Anal.},
      volume={153},
       pages={142\ndash 168},
         url={https://doi.org/10.1016/j.na.2016.08.027},
      review={\MR{3614666}},
}

\bib{BH24}{article}{
      author={Bui, The~Anh},
      author={Hisa, Kotaro},
       title={Existence of solutions semilinear parabolic equations with
  singular initial data in the {H}eisenberg group},
        date={2024},
     journal={arXiv:2408.16985},
}

\bib{CKS10}{article}{
      author={Chen, Zhen-Qing},
      author={Kim, Panki},
      author={Song, Renming},
       title={Heat kernel estimates for the {D}irichlet fractional
  {L}aplacian},
        date={2010},
        ISSN={1435-9855,1435-9863},
     journal={J. Eur. Math. Soc. (JEMS)},
      volume={12},
      number={5},
       pages={1307\ndash 1329},
         url={https://doi.org/10.4171/JEMS/231},
      review={\MR{2677618}},
}

\bib{Cho06}{article}{
      author={Cho, Sungwon},
       title={Two-sided global estimates of the {G}reen's function of parabolic
  equations},
        date={2006},
        ISSN={0926-2601,1572-929X},
     journal={Potential Anal.},
      volume={25},
      number={4},
       pages={387\ndash 398},
         url={https://doi.org/10.1007/s11118-006-9026-0},
      review={\MR{2255354}},
}

\bib{DPV12}{article}{
      author={Di~Nezza, Eleonora},
      author={Palatucci, Giampiero},
      author={Valdinoci, Enrico},
       title={Hitchhiker's guide to the fractional {S}obolev spaces},
        date={2012},
        ISSN={0007-4497,1952-4773},
     journal={Bull. Sci. Math.},
      volume={136},
      number={5},
       pages={521\ndash 573},
         url={https://doi.org/10.1016/j.bulsci.2011.12.004},
      review={\MR{2944369}},
}

\bib{EG15}{book}{
      author={Evans, Lawrence~C.},
      author={Gariepy, Ronald~F.},
       title={Measure theory and fine properties of functions},
     edition={Revised},
      series={Textbooks in Mathematics},
   publisher={CRC Press, Boca Raton, FL},
        date={2015},
        ISBN={978-1-4822-4238-6},
      review={\MR{3409135}},
}

\bib{FK13}{article}{
      author={Felsinger, Matthieu},
      author={Kassmann, Moritz},
       title={Local regularity for parabolic nonlocal operators},
        date={2013},
        ISSN={0360-5302,1532-4133},
     journal={Comm. Partial Differential Equations},
      volume={38},
      number={9},
       pages={1539\ndash 1573},
         url={https://doi.org/10.1080/03605302.2013.808211},
      review={\MR{3169755}},
}

\bib{FK12}{article}{
      author={Fino, Ahmad~Z.},
      author={Kirane, Mokhtar},
       title={Qualitative properties of solutions to a time-space fractional
  evolution equation},
        date={2012},
        ISSN={0033-569X,1552-4485},
     journal={Quart. Appl. Math.},
      volume={70},
      number={1},
       pages={133\ndash 157},
         url={https://doi.org/10.1090/S0033-569X-2011-01246-9},
      review={\MR{2920620}},
}

\bib{FHIL23}{article}{
      author={Fujishima, Yohei},
      author={Hisa, Kotaro},
      author={Ishige, Kazuhiro},
      author={Laister, Robert},
       title={Solvability of superlinear fractional parabolic equations},
        date={2023},
        ISSN={1424-3199,1424-3202},
     journal={J. Evol. Equ.},
      volume={23},
      number={1},
       pages={Paper No. 4, 38},
         url={https://doi.org/10.1007/s00028-022-00853-z},
      review={\MR{4520263}},
}

\bib{FHIL24}{article}{
      author={Fujishima, Yohei},
      author={Hisa, Kotaro},
      author={Ishige, Kazuhiro},
      author={Laister, Robert},
       title={Local solvability and dilation-critical singularities of
  supercritical fractional heat equations},
        date={2024},
        ISSN={0021-7824,1776-3371},
     journal={J. Math. Pures Appl. (9)},
      volume={186},
       pages={150\ndash 175},
         url={https://doi.org/10.1016/j.matpur.2024.04.005},
      review={\MR{4745503}},
}

\bib{GK01}{article}{
      author={Gidda, M.},
      author={Kirane, M.},
       title={Criticality for some evolution equations},
        date={2001},
        ISSN={0374-0641},
     journal={Differ. Uravn.},
      volume={37},
      number={4},
       pages={511\ndash 520, 574\ndash 575},
         url={https://doi.org/10.1023/A:1019283624558},
      review={\MR{1854046}},
}

\bib{HI18}{article}{
      author={Hisa, Kotaro},
      author={Ishige, Kazuhiro},
       title={Existence of solutions for a fractional semilinear parabolic
  equation with singular initial data},
        date={2018},
        ISSN={0362-546X},
     journal={Nonlinear Anal.},
      volume={175},
       pages={108\ndash 132},
         url={https://doi.org/10.1016/j.na.2018.05.011},
      review={\MR{3830724}},
}

\bib{HIT20}{article}{
      author={Hisa, Kotaro},
      author={Ishige, Kazuhiro},
      author={Takahashi, Jin},
       title={Existence of solutions for an inhomogeneous fractional semilinear
  heat equation},
        date={2020},
        ISSN={0362-546X,1873-5215},
     journal={Nonlinear Anal.},
      volume={199},
       pages={111920, 28},
         url={https://doi.org/10.1016/j.na.2020.111920},
      review={\MR{4099496}},
}

\bib{HM24}{article}{
      author={Hisa, Kotaro},
      author={Miyamoto, Yasuhito},
       title={Threshold property of a singular stationary solution for
  semilinear heat equations with exponential growth},
        date={2024},
     journal={arXiv:2409.16549},
}

\bib{IS19}{article}{
      author={Ikeda, Masahiro},
      author={Sobajima, Motohiro},
       title={Sharp upper bound for lifespan of solutions to some critical
  semilinear parabolic, dispersive and hyperbolic equations via a test function
  method},
        date={2019},
        ISSN={0362-546X,1873-5215},
     journal={Nonlinear Anal.},
      volume={182},
       pages={57\ndash 74},
         url={https://doi.org/10.1016/j.na.2018.12.009},
      review={\MR{3894246}},
}

\bib{IKO20}{article}{
      author={Ishige, Kazuhiro},
      author={Kawakami, Tatsuki},
      author={Okabe, Shinya},
       title={Existence of solutions for a higher-order semilinear parabolic
  equation with singular initial data},
        date={2020},
        ISSN={0294-1449,1873-1430},
     journal={Ann. Inst. H. Poincar\'e{} C Anal. Non Lin\'eaire},
      volume={37},
      number={5},
       pages={1185\ndash 1209},
         url={https://doi.org/10.1016/j.anihpc.2020.04.002},
      review={\MR{4138231}},
}

\bib{Ju05}{article}{
      author={Ju, Ning},
       title={The maximum principle and the global attractor for the
  dissipative 2{D} quasi-geostrophic equations},
        date={2005},
        ISSN={0010-3616,1432-0916},
     journal={Comm. Math. Phys.},
      volume={255},
      number={1},
       pages={161\ndash 181},
         url={https://doi.org/10.1007/s00220-004-1256-7},
      review={\MR{2123380}},
}

\bib{KQ02}{article}{
      author={Kirane, Mokhtar},
      author={Qafsaoui, Mahmoud},
       title={Global nonexistence for the {C}auchy problem of some nonlinear
  reaction-diffusion systems},
        date={2002},
        ISSN={0022-247X,1096-0813},
     journal={J. Math. Anal. Appl.},
      volume={268},
      number={1},
       pages={217\ndash 243},
         url={https://doi.org/10.1006/jmaa.2001.7819},
      review={\MR{1893203}},
}

\bib{LPPS15}{article}{
      author={Leonori, Tommaso},
      author={Peral, Ireneo},
      author={Primo, Ana},
      author={Soria, Fernando},
       title={Basic estimates for solutions of a class of nonlocal elliptic and
  parabolic equations},
        date={2015},
        ISSN={1078-0947,1553-5231},
     journal={Discrete Contin. Dyn. Syst.},
      volume={35},
      number={12},
       pages={6031\ndash 6068},
         url={https://doi.org/10.3934/dcds.2015.35.6031},
      review={\MR{3393266}},
}

\bib{QS19}{book}{
      author={Quittner, Pavol},
      author={Souplet, Philippe},
       title={Superlinear parabolic problems. {Blow}-up, global existence and
  steady states},
     edition={2nd revised and updated edition},
      series={Birkh{\"a}user Adv. Texts, Basler Lehrb{\"u}ch.},
   publisher={Cham: Birkh{\"a}user},
        date={2019},
        ISBN={978-3-030-18220-5; 978-3-030-18222-9},
}

\bib{R05}{article}{
      author={Riahi, Lotfi},
       title={Comparison of {G}reen functions and harmonic measures for
  parabolic operators},
        date={2005},
        ISSN={0926-2601,1572-929X},
     journal={Potential Anal.},
      volume={23},
      number={4},
       pages={381\ndash 402},
         url={https://doi.org/10.1007/s11118-005-2606-6},
      review={\MR{2139572}},
}

\bib{TY23}{article}{
      author={Takahashi, Jin},
      author={Yamamoto, Hikaru},
       title={Solvability of a semilinear heat equation on {R}iemannian
  manifolds},
        date={2023},
        ISSN={1424-3199,1424-3202},
     journal={J. Evol. Equ.},
      volume={23},
      number={2},
       pages={Paper No. 33, 55},
         url={https://doi.org/10.1007/s00028-023-00883-1},
      review={\MR{4578482}},
}

\bib{Z02}{article}{
      author={Zhang, Qi~S.},
       title={The boundary behavior of heat kernels of {D}irichlet
  {L}aplacians},
        date={2002},
        ISSN={0022-0396,1090-2732},
     journal={J. Differential Equations},
      volume={182},
      number={2},
       pages={416\ndash 430},
         url={https://doi.org/10.1006/jdeq.2001.4112},
      review={\MR{1900329}},
}

\end{biblist}
\end{bibdiv}

\end{document}